\documentclass[11pt,reqno]{amsart}
\usepackage{graphicx,tikz}
\usepackage{amsfonts,amsmath,amssymb}
\usetikzlibrary{matrix,arrows}
\newcommand{\cx}{{\mathbb{C}}}

\newcommand{\D}{\mathbb D}
\newcommand{\C}{\mathbb C}
\newcommand{\R}{\mathbb R}
\newcommand{\N}{\mathbb N}

\newtheorem{theorem}{Theorem}[section]
\newtheorem{lemma}[theorem]{Lemma}
\newtheorem{prop}[theorem]{Proposition}
\newtheorem{cor}[theorem]{Corollary}

\theoremstyle{definition}

\begin{document}
\title{On dicritical singularities of Levi-flat sets}
\author{Sergey Pinchuk*, Rasul Shafikov** and Alexandre Sukhov***}
\begin{abstract}
It is proved that dicritical singularities of real analytic Levi-flat sets coincide with the set
of Segre degenerate points.
\end{abstract}

\maketitle

\let\thefootnote\relax\footnote{MSC: 37F75,34M,32S,32D.
Key words: Levi-flat  set, dicritical singularity, foliation, Segre variety.
}

* Department of Mathematics, Indiana University, 831 E 3rd St. Rawles Hall, Bloomington,   IN 47405, USA, e-mail: pinchuk@indiana.edu

** Department of Mathematics, the University of Western Ontario, London, Ontario, N6A 5B7, Canada,
e-mail: shafikov@uwo.ca. The author is partially supported by the Natural Sciences and Engineering 
Research Council of Canada.

***Universit\'e  de Lille (Sciences et Technologies), 
U.F.R. de Math\'ematiques, 59655 Villeneuve d'Ascq, Cedex, France,
e-mail: sukhov@math.univ-lille1.fr. The author is partially supported by Labex CEMPI.

\section{Introduction}

A real analytic Levi-flat set $M$ in $\C^{N}$ is a real analytic set such that its regular part is  a Levi-flat CR manifold of hypersurface type. 
An important special case (closely related to the theory of holomorphic foliations) arises when $M$ is a hypersurface. The local geometry of a Levi-flat hypersurface near its singular locus has been studied by several authors \cite{Bru1, Bru2, CerNeto, CerSad, Fe, Le, PSS, SS}. One of the main questions here concerns an extension of the Levi foliation of the regular part of $M$ as a (singular) holomorphic foliation (or, more generally, a singular holomorphic web) to a full neighbourhood of a singular point. The  existence of such an extension allows one  to use the holomorphic resolution of singularities results for the study of local geometry of singular Levi-flat hypersurfaces. 

The present paper is concerned with local properties of real analytic Levi-flat sets near their singularities. These sets arise in the study of Levi-flat hypersurfaces when lifted to the projectivization of the cotangent bundle of the ambient space. Our main result gives a complete characterization of dicritical singularities of such sets in terms of their Segre varieties.  Our method is a straightforward
generalization of arguments in~\cite{PSS,SS} where the case of Levi-flat hypersurfaces is considered.

\section{ Levi-flat subsets,  Segre varieties}

In this section we provide relevant background material on  real analytic Levi-flat  sets (of higher codimension)  and their Segre varieties. To the best of our knowledge this topic has not been considered in detail before; for 
convenience of the reader we provide some details.

\subsection{Real and complex analytic sets.}

Let $\Omega$ be a domain in $\C^{N}$. We denote by $z = (z_1,...,z_{N})$ the standard complex coordinates. A closed subset $M \subset \Omega$ is called {\it a real (resp. complex) analytic subset} in 
$\Omega$ if it is locally defined by a finite collection of real analytic (resp. holomorphic) functions. 

For a real analytic $M$ this means that for every point $q \in \Omega$ there exist a neighbourhood $U$ of $q$ and real analytic vector function 
$\rho = (\rho_1,...,\rho_k): U \to \R^k$ such that 
\begin{eqnarray}
\label{DefEq}
M \cap U = \rho^{-1}(0) = \{z \in U: \rho_j(z,\overline{z}) = 0,\ j = 1,...,k \} .
\end{eqnarray}
In fact, one can reduce the situation to the case $k=1$ by considering the defining function $\rho_1^2 + ...+ \rho_k^2$. Without loss of generality assume $q = 0$ and choose a neighbourhood $U$ in (\ref{DefEq}) in the form of a polydisc  
$\Delta(\varepsilon) = \{ z \in \C^{N}: \vert z_j \vert < \varepsilon \}$
 of radius $ \varepsilon > 0$.  Then, for $\varepsilon$ small enough, the (vector-valued) function $\rho$ admits the  Taylor  expansion 
 convergent in $U$:
\begin{eqnarray}\label{exp}
\rho(z,\overline z) = \sum_{IJ} c_{IJ}z^I \overline{z}^J, \ c_{IJ}\in\cx, \ \ I,J \in \N^N.
\end{eqnarray}
Here and below we use the multi-index notation $I = (i_1,....,i_N)$ and $\vert I \vert = i_1 +...+i_N$. The ($\C^k$-valued) coefficients $c_{IJ}$ satisfy the condition
\begin{eqnarray}
\label{coef}
\overline c_{IJ} = c_{JI},
\end{eqnarray}
since $\rho$ is a real ($\R^k$-valued) function.

An analytic subset $M$ is called {\it irreducible} if it cannot be represented as a union $M = M_1 \cup M_2$ where $M_j$ are analytic subsets of $\Omega$ different from $M$.  Similarly, an analytic subset is irreducible as a germ at a point $p \in M$ if its germ cannot be represented as a union of  germs of two real analytic sets. All considerations of the present paper are local and we 
always assume irreducibility of germs even if it is not specified explicitly.

A set $M$ can be decomposed into a disjoint union $M = M_{reg} \cup M_{sing}$, the regular and the singular part respectively. The regular part $M_{reg}$ is a nonempty and open subset of $M$. In the real analytic case we adopt the following convention: $M$ is a real analytic submanifold of maximal dimension in a neighbourhood of every point of $M_{reg}$. This dimension is called the dimension of $M$ and is denoted by $\dim M$. The set $M_{sing}$ is a real semianalytic subset of 
$\Omega$ of dimension $< \dim M$. Unlike complex analytic sets, for a real analytic~$M$, the set $M_{sing}$ may contain manifolds of smaller dimension which are not in the closure of $M_{reg}$, as seen in the classical example of the Whitney umbrella. Therefore, in general $M_{reg}$ is not dense in $M$.

Recall that the dimension of a complex analytic set $A$ at a point $a \in A$ is defined as 
$$
\dim_a A : = \underset{{A_{reg} \ni z \to a}}{\overline{\lim}} \dim_z A ,
$$ 
and that the the function $z \mapsto \dim_z A$ is upper semicontinuous. Suppose that $A$ is an irreducible  complex analytic subset of a domain $\Omega$ and let $F: A \to X$  be a  holomorphic mapping into some complex manifold $X$. The local dimension of $F$ at a point $z\in A$ 
is defined as $\dim_{z} F = \dim A - \dim_{z} F^{-1}( F(z))$ and the dimension of $F$ is set to be $\dim F = \max_{z \in A}\dim_{z} F$. Note that the equality $\dim_{z} F = \dim F$ holds on a Zariski open subset of 
$A$, and that  $\dim F$ coincides with the rank of the map~$F$, see \cite{Ch}.

\subsection{Complexification and Segre varieties.}
Let $M$ be the germ at the origin of an irreducible real analytic subset of $\C^{N}$ defined by~\eqref{exp}. 
We are interested in the geometry of $M$ in an arbitrarily small neighbourhood of $0$. We may consider a sufficiently 
small open neighbourhood $U$ of the origin and a representative of the germ which is also irreducible, see \cite{N} for 
details. In what follows we will not distinguish between the germ of $M$ and its particular representative in a suitable 
neighbourhood of the origin.

Denote by $J$ the standard complex structure of $\C^{N}$ and consider the opposite structure $-J$.
Consider the space $\C^{2N}_\bullet := (\C^{N}_z,J) \times (\C^{N}_w,-J)$ and the diagonal
$$
\Delta = \left\{ (z,w) \in \C^{2N}_\bullet : z=w \right\} .
$$
The set $M$ can be lifted to $\C^{2N}_\bullet$ as the real analytic subset 
$$
\hat{M} := \left\{ (z,z) \in \C^{2N}_\bullet : z \in M \right\} .
$$
There exists a unique irreducible complex analytic subset $M^{\C}$ in $\C^{2N}_\bullet$ of complex dimension equal
to the real dimension of $M$ such that 
$\hat{M} = M^{\C} \cap \Delta$ (see \cite{N}). The set $M^{\C}$ is called the {\it complexification} of $M$. The 
antiholomorphic involution 
$$
\tau: \C^{2N}_\bullet \to \C^{2N}_\bullet, \,\,\, \tau: (z,w) \mapsto (w,z)
$$  
leaves $M^\C$ invariant and $\hat{M}$ is the set of fixed points of $\tau\vert_{M^{\mathbb C}}$.

The complexification $M^\C$ is equipped with two canonical  holomorphic projections $\pi_z:(z,w) \mapsto z$ and $\pi_w:(z,w) \mapsto w$. We always suppose by convention that the domain of these projections is $M^\C$. The triple $(M^\C,\pi_z,\pi_w)$ is represented by the following diagram 

\begin{center}
\begin{tikzpicture}[description/.style={fill=white,inner sep=2pt}]
    \matrix (m) [matrix of math nodes, row sep=3em,
    column sep=2em, text height=2ex, text depth=0.25ex]
    {  & M^{\C} & \\
     (\C^{N},J)  & & (\C^{N},-J) \\
    };
       \path[->] (m-1-2) edge  node[auto,swap] {$ \pi_z $}  (m-2-1);
       \path[->] (m-1-2) edge node[auto] {$ \pi_w $} (m-2-3);
\end{tikzpicture}
\end{center}
which leads to the central notion of the present paper. The {\it Segre variety} of a point $w \in \C^{N}$ 
is defined as 
$$
Q_w := (\pi_z \circ \pi^{-1}_w)(w) = \left\{ z \in \C^{N}: (z,w) \in M^\C\right\} .
$$
When $M$ is a hypersurface defined by (\ref{DefEq}) (with $k = 1$) this definition coincides with the usual definition $$Q_w = \left\{ z: \rho(z,\overline{w}) = 0 \right\}.$$
Of course, here we suppose that $\rho$ is a minimal function, that is, it generates the ideal of real analytic functions vanishing on $M$.

The following properties of Segre varieties are well-known for hypersurfaces.

\begin{prop}\label{SegreProp}
Let $M$ be the germ of a real analytic subset in $\mathbb C^N$. Then
\begin{itemize}
\item[(a)] $z \in Q_z \Longleftrightarrow z \in M$.
\item[(b)] $z \in Q_w \Longleftrightarrow w \in Q_z$,
\item[(c)] (invariance property) Let $M_1$ be a real analytic CR manifold, and $M_2$ be a real analytic germ 
in $\C^N$ and $\C^K$ respectively. Let  $p \in  M_1$, $q\in M_2$, and $U_1 \ni p$, $U_2 \ni q$ be small neighbourhoods. 
Let also 
$f:U_1 \to U_2$ be a holomorphic map such that $f(M_1 \cap U_1) \subset M_2 \cap U_2$. Then
$$f(Q_w^1) \subset Q^2_{f(w)}$$
for all $w$ close to $p$. If, in addition, $M_2$ is nonsingular and $f:U_1 \to U_2$ is biholomorphic, then $f(Q_w^1) = Q_{f(w)}^2$.
Here $Q_w^1$ and $Q_{f(w)}^2$ are Segre varieties associated with $M_1$ and $M_2$ respectively.
\end{itemize}
\end{prop}
\begin{proof} (a) Note that  $z \in Q_z$ if and only if $(z,z) \in M^\C$, which is equivalent to $(z,z) \in \hat M$.

(b) The relation $(z,w) \in M^\C$ holds if and only if $\tau(z,w) = (w,z) \in M^\C$.

(c) Suppose that $M_1$ is defined near $p$ by the equations $\rho_1 = ...=\rho_k = 0$ and 
$d\rho_1 \wedge ... \wedge d\rho_k \neq 0$. Similarly, suppose that $M_2$ is defined by the equations 
$\phi_1 = ... = \phi_l = 0$.  
Then the Segre varieties are respectively given by 
$Q_w^1 = \{ z: \rho_j(z,\overline{w}) = 0,\ j = 1,...,k\}$ and $Q_w^2 = \{ z: \phi_s (z,\overline{w}) = 0, \ s=1,...,l \}$. By assumption we have 
$\phi_s(f(z),\overline{f(z)}) = 0$ when $z \in M_1$. This implies that there exist real analytic functions $\lambda_j$, $j=1,...,k$, such that 
$$
\phi_s(f(z),\overline{f(z)}) = \sum_1^k \lambda_{sj}(z,\overline{z})\rho_j(z,\overline{z}) .
$$

Consider first the case where $M_1$ is a generic manifold, that is, $\partial \rho_1 \wedge ... \wedge \partial \rho_k\neq 0$.
Let $f^*$ be a holomorphic function such that $f^*(\overline{w}) = \overline{f(w)}$. Since $M_1$ is generic, it is the 
uniqueness set for holomorphic functions. Therefore, by analyticity we have
$$
\phi_s(f(z),f^*(\overline{w})) = \sum_1^k \lambda_{sj}(z,\overline{w})\rho_j(z,\overline{w}),
$$
for all $z$ and $w$. This concludes the proof  for the case when $M_1$ is generic.

Let now $M_1$ be a CR manifold which is not generic. Then, by real analyticity, $M_1$ can be represented as the graph of (the restriction of) a holomorphic (vector) function over a real analytic generic manifold $\tilde M_1$ of real codimension $l$ in $\C^d$, for some $l$ and $d$. More precisely, set $z = (z',z'')$, $z' = (z_1,...,z_d)$, $z'' = (z_{d+1},...,z_N)$.
Then $\tilde M_1 = \{ z': \psi_j(z',\overline{ z'}) =0,\ j=1,...,l \}$ and $M_1 = \{ (z',z''): z' \in N_1, z'' = g(z') \}$, where $\psi_j$ are real analytic functions and $g$ is a holomorphic (vector) function. Then every Segre variety $Q_w^1$ of $M_1$ is the graph of $g$ over the Segre variety of $\tilde M_1$. Indeed,
$Q_w^1 = \{ (z',z''): \phi_j(z',\overline{w'}) = 0, j= 1,...,l, z'' = g(z') \}$. The holomorphic map $\tilde f(z') = f(z',g(z'))$ transforms the generating manifold $\tilde M_1$ to $M_2$. Since we already proved the result for generic submanifolds in the 
source, we conclude that the map $\tilde f$ transforms Segre varieties of the  manifold $\tilde M_1$ to Segre varieties of the manifold 
$M_2$. This implies the required statement.
\end{proof}

\subsection{Levi-flat sets}

We say that an irreducible real analytic set $M\subset \mathbb C^{n+m}$ is {\it Levi-flat} if $\dim M = 2n-1$ and $M_{reg}$ is locally foliated by complex manifolds of complex dimension $n-1$. In particular, $M_{reg}$ is a CR manifold of hypersurface type. 
The most known case arises when  $m=0$, i.e., when $M$ is a Levi-flat hypersurface in $\C^n$. 

We use the notation $z'' = (z_{n+1},...,z_{n+m})$, and similarly for the $w$ variable.
It follows from the Frobenius theorem and the implicit function theorem that for every 
point  $q \in M_{reg}$ there exist an open neighbourhood $U$ and a local biholomorphic 
change of coordinates $F: (U,q) \to (F(U),0)$ such that $F(M)$ has the form 
\begin{eqnarray}
\label{NormForm1}
\{ z \in F(U): z_n + \overline{z}_n = 0, \ z'' = 0 \} .
\end{eqnarray}
The subspace $F(M)$ is foliated by complex affine subspaces $L_c = \{z_n= ic,\ z'' = 0, \ c \in  \R\}$,
which gives a foliation of $ M_{reg} \cap U$ by complex submanifolds $F^{-1}(L_c)$. This defines a 
foliation on $M_{reg}$ which is called {\it the Levi foliation} and denoted  by~$\mathcal L$. Every leaf of $\mathcal L$ is 
tangent to the complex tangent space of $M_{reg}$.  The complex affine subspaces 
\begin{eqnarray}
\label{SegrePlane}
\left\{  z_n = c, z'' = 0 \right\}, \,\,\ c \in \C, 
\end{eqnarray}
in local coordinates given by
(\ref{NormForm1}) are precisely the Segre varieties of $M$ for every complex $c$. Thus, the Levi foliation is closely related to 
Segre varieties. The complexification $M^\C$ is given by 
\begin{eqnarray}\label{compl}
M^\C = \{ (z,w): z_n + \overline{w}_n = 0, z'' =0, w'' = 0 \}.
\end{eqnarray}
For $M$ defined by (\ref{NormForm1}) its Segre varieties (\ref{SegrePlane}) 
fill the complex subspace $z'' = 0$ of $\C^{n+m}$. In particular, if $w$ is not in this subspace, then 
$Q_w$ is empty. 

In arbitrary coordinates, in a neighbourhood $U \subset \C^{n+m}(z)$ of a regular point $z^0 \in M$ the Levi flat set is the transverse intersection of a real analytic hypersurface with a complex n-dimensional manifold, that is 
\begin{eqnarray}
\label{compl1}
M =  \{ z \in U: h_j(z) = 0, \ j=1,...,m , \,\,  r(z,\overline z) = 0 \} .
\end{eqnarray}
Here $h_j$ are functions holomorphic on $U$ and $r:U \longrightarrow \R$ is a real analytic function.
Furthermore, $\partial r \wedge dh_1 \wedge...\wedge dh_m \neq 0$. Then 

\begin{eqnarray}
\label{compl2}
M^\C = \{ (z,w) \in U \times U: r(z,\overline w ) = 0, h_j(z) = 0, h_j(\overline w) = 0, j=1,...,m \}
\end{eqnarray}
in a neighbourhood $U \times U$ of $(z^0,{\overline z}^0)$.

We need to study some general properties of projections $\pi_z$ and $\pi_w$. Let $\pi$ be one of the projections $\pi_z$ or $\pi_w$. Introduce 
the dimension of $\pi$ by setting $\dim \pi = \max_{(z,w) \in M^\C}\dim_{(z,w)} \pi$.
If $M$ is irreducible as a germ, then so is $M^\C$ (see~\cite[p.92]{N}). Hence, $(M^\C)_{reg}$ is a connected complex manifold of dimension $2n-1$. Then the equality $\dim_{(z,w)} \pi = \dim \pi$ holds on a Zariski open set 
\begin{equation}\label{e.mstarc}
M^\C_*: = (M^\C \setminus X) \subset (M^\C)_{reg},
\end{equation}
where $X$ is a complex analytic subset of dimension $< 2n-1$. Here $\dim \pi$ coincides with the rank of $\pi \vert_{M^\C_*}$.
Furthermore, $\dim (\pi \vert (M^\C)_{sing}) \le \dim \pi$.

\begin{lemma}
\label{dim1}
Let $\pi$ be one of the projections $\pi_z$ or $\pi_w$.
\begin{itemize}
\item[(a)] We have $\dim \pi = n$.
\item[(b)] The image $\pi(M^\C)$ is contained in the (at most) countable union of complex analytic sets of dimension $\le n$.
\end{itemize}
\end{lemma}
\begin{proof}
(a) Consider the case where $\pi = \pi_w$. In view of (\ref{compl2}) 
the image  of an open  neighbourhood a regular point 
$(z^0,{\overline z}^0)$ in $M^\C$  coincides with the  complex $n$-dimensional manifold  
$\{w:  h_j(\overline w) = 0, j=1,...,m \}$. This implies (a).

(b) This is a consequence of (a), see \cite{Ch}.
\end{proof}

It follows from the lemma above that generically, i.e., for $w \in \pi_w(M^\C_*)$, the complex analytic set $Q_w$ has 
dimension $n-1$, and that $Q_w$ can have dimension $n$ if $(z,w)\notin M^\C_*$. 
Of course, $Q_w$ is empty if $w$ does not belong to  $ \pi_w(M^\C)$.

A singular point $q \in M$ is called {\it Segre degenerate} if $\dim Q_q = n$. 
Note that the set of Segre degenerate points is contained in a complex analytic subset of dimension $n-2$. The proof,
which we omit, is quite similar to that in \cite{SS}, where this claim is established for hypersurfaces.

Let $q \in M_{reg}$. Denote by ${\mathcal L}_q$ the leaf of the Levi foliation through $q$. Note that
by definition this is a connected complex submanifold of complex dimension $n-1$ that is closed in $M_{reg}$. 
Denote by $M_* \subset M_{reg}$ the image of $\hat M \cap M^\C_*$ under the projection $\pi$,
where $M^\C_*$ is defined as in~\eqref{e.mstarc}. This set coincides with $M_{reg} \setminus A$ for some proper 
complex analytic subset $A$.

As a simple consequence of Proposition \ref{SegreProp} we have quite similarly 
to~\cite{SS} the following.

\begin{cor}
\label{Segre+Levi}
Let  $a \in M_*$. Then the following holds:
\begin{itemize}
\item[(a)]  The leaf ${\mathcal L}_a$ is contained in a unique irreducible component $S_a$ of 
$Q_a$ of dimension $n-1$. In particular, $Q_a$ is a nonempty complex analytic set of pure dimension $n-1$. In a small neighbourhood $U$ of $a$ the intersection $S_a \cap U$ is also a unique complex submanifold of complex dimension $n-1$ through $a$ which is contained in $M$.
\item[(b)] For every $a \in M_*$ the complex variety $S_a$ is contained in $M$;
\item[(c)] For every $a, b \in M_*$ one has $b \in S_a \Longleftrightarrow S_a = S_b$.
\item[(d)] Suppose that $a \in M_*$ and ${\mathcal L}_a$ touches a point  $q \in M$  (the point $q$  may be singular). Then  $Q_q$ contains $S_a$. If, additionally, $\dim_\C Q_q = n-1$, then $S_a$ is  an irreducible component of~$Q_q$. 
\end{itemize}
\end{cor}

\begin{proof} 
(a) We first make an elementary but important observation. Suppose that $M$ is a representative in a domain $U$ of the germ 
of a real analytic set $\{\rho = 0\}$. Let $a\in M$ and $V$ be a neighbourhood of $a$, $V\subset U$. Suppose that we used a different function $\tilde \rho$ to define $M\cap V$, i.e., $M\cap V = \{\tilde\rho = 0\}$. Applying Proposition~\ref{SegreProp} 
to the inclusion map $V \hookrightarrow U$ we conclude that the Segre varieties of $M\cap V$ defined by complexifying $\tilde\rho$
are contained in the Segre varieties of $M$ defined by complexifying the function $\rho$. More precisely, the Segre varieties with
respect to $\tilde\rho$ are coincide with the intersection with $V$ of some components of Segre varieties with respect to $\rho$.

In view of the invariance of the Levi form under biholomorphic maps, the Levi foliation is an intrinsic notion, i.e., it is
independent of the choice of (local) holomorphic coordinates. Similarly, in view of the biholomorphic invariance of Segre 
varieties described in Proposition~\ref{SegreProp}(c), these are also intrinsic objects. There exist a small neighbourhood 
$U$ of $a$ and a holomorphic map  which takes $a$ to the origin and is one-to-one between $U$ and a neighbourhood $U'$ 
of the origin, such that the image of $M$ has the form (\ref{NormForm1}). Hence, without loss of generality we may assume  
that $a = 0$ and may view (\ref{NormForm1}) as a representation of $M \cap U$ in the above local coordinates. Then 
$Q_0 \cap U = \{ z_n = 0, z'' = 0 \}$. Hence, going back to the original coordinates, we obtain, by the invariance of Segre 
varieties, that the intersection $Q_a \cap U$ is a  complex submanifold of dimension $n-1$ in $M \cap U$ which coincides 
with ${\mathcal L}_a \cap U$. In particular, it belongs to a unique irreducible component of $Q_a$ of dimension $n-1$. It follows also from (\ref{NormForm1}) that it is a unique complex submanifold of dimension $n-1$ through $a$ contained in a 
neighbourhood of $a$ in $M$.

(b) Recall that we consider $M$ defined by (\ref{DefEq}). Since $S_a$ is contained in $M$ near $a$, it follows by analyticity of $\rho$ and the uniqueness  that 
$\rho \vert_{S_a} \equiv 0$, i.e., $S_a$ is contained in $M$.

(c) By part (b), the complex submanifold $S_a$ is contained in $M$. Therefore, in a small neighbourhood of  $b$ we have $S_a = S_b$ 
by part (a). Then also globally $S_a = S_b$ by the uniqueness theorem for irreducible complex analytic sets.

(d) Since $q \in Q_a$, we have $a \in Q_q$. The same holds for every point $a'\in{\mathcal L}_a$ in a neighbourhood of $a$. 
Hence, $S_a$ is contained in $Q_q$ by the uniqueness theorem for complex analytic sets. Suppose now that 
$\dim_{\C} Q_q = n-1$.  Since $a$ is a regular point of $M$,  the leaf ${\mathcal L}_a$ is not contained in the set of singular points of $M$; hence, regular points of 
$M$ form an open dense subset in this leaf. Consider a sequence of points 
$q^m \in {\mathcal L}_a \cap M_{reg}$ converging to $q$. It follows by (c) that the complex 
$n-1$-submanifold  $S_a = S_{q^m}$ is independent of $m$  
and by (a) $S_{q^m}$ is an irreducible component of $Q_{q^m}$. Passing to the limit we obtain that  $S_a$ is contained in $Q_q$ as an 
irreducible component.
\end{proof}


\section{Dicritical singularities of Levi-flat subsets}

Let $M$ be a real analytic Levi-flat subset of dimension $2n-1$  in $ \C^{n+m}$. A singular point $q\in M$ is called 
{\it dicritical} if $q$ belongs to  infinitely many geometrically different leaves ${\mathcal L}_a$. 
Singular points which are not dicritical are called {\it nondicritical}.  
Our main result is the following.

\begin{theorem}\label{DicritTheo}
Let $M$ be a real analytic Levi-flat subset of dimension $2n-1$ in  $\C^{n+m}$, irreducible as a germ at  $0 \in \overline{M_{reg}}$ . Then $0$ is a dicritical point if and only if  $\dim_\C Q_0 = n$.
\end{theorem}

For hypersurfaces this result is obtained in \cite{PSS}.

 A dicritical point is Segre degenerate; this  follows by Corollary ~\ref{Segre+Levi}(d).  
 From now on we assume that  $0$ is a Segre degenerate point; our goal is to prove that $0$ is a  dicritical point.

For every point $w \in \pi_w(M^\C)$ denote by $Q_w^c$ the union of irreducible components of $Q_w$ containing the origin; we call it {\it the canonical Segre variety}. Note that by  Proposition~\ref{SegreProp}(b), for every $w$ from a neighbourhood of the origin in $Q_0$ its canonical Segre variety $Q_w^c$ is a nonempty complex analytic subset.  Consider the set 
$$\Sigma = \{ (z,w) \in  M^\C_*: z  \notin Q_{w}^c \} .$$

If  $\Sigma$ is empty, then for every point $w$ from a neighbourhood of the origin in $\pi_w(M^\C_*)$  the Segre variety $Q_{w}$ coincides 
with the canonical Segre variety $Q_w^c$, i.e., all components of $Q_w$ contain the origin. But for  a regular point $w$ 
of $M$, its Levi leaf is  a component of $Q_{w}$. Therefore, every Levi leaf contains the origin which is then necessarily a dicritical point. Thus, in order to prove the theorem, it suffices to establish the following

\begin{prop} 
\label{empty}
$\Sigma$ is the empty set.
\end{prop}
 Arguing by contradiction assume that $\Sigma$ is not empty. The proof consists in two main steps. First, we prove that the boundary of $\Sigma$ is ``small enough", and  so is a removable singularity for $\Sigma$. Second, we prove that the complement of $\Sigma$ is not empty. This will lead to a contradiction.

To begin, we need some technical preliminaries. Consider the complex $2n+m$ dimensional analytic set 
$$
Z = \C^{n+m} \times Q_0 = \{ (z,w): w \in Q_0 \} .
$$ 
Here we view a copy of $Q_0$ in $\C^{n+m}(w)$ that is defined by $\pi_w \circ \pi_z^{-1}(0)$.

\begin{lemma}
\label{Q0}
One has $M^\C \subset Z$.  As a consequence,  $0 \in Q_w$ for every $(z,w) \in M^\C$.
\end{lemma}
Essentially this result was proved by Brunella \cite{Bru1}. For the convenience of readers we include the proof.
\begin{proof} Denote by $X$  the proper complex analytic subset of $M^\C$ where the dimension of fibres of $\pi_w$ is $\ge n$.  
Thus for every  $(z,w) \in M^\C \setminus X$ the dimension of the fibres  $\pi_w^{-1}(w)$ is equal to $n-1$. 

Note that the lift $\tilde Q_0 = \{ (z,w): z= 0, w \in Q_0 \}$ is contained in $M^\C$. First we claim that the  intersection $\tilde Q_0 \cap (M^\C \setminus X)$ is not empty. Arguing by contradiction, assume that $\tilde Q_0$ is contained in $X$. Then   the dimension of the fibre of $\pi_w$ at every point of $\tilde Q_0$ is $\ge n$. Since $\dim \tilde Q_0 = n$, the dimension of $M^\C$  must be $\ge 2n$, which is a contradiction.  

Let $(0,w^0) \in  M^\C \setminus X$. Slightly perturbing $w^0$ one can assume that $w^0$ is a regular point of $Q_0$. Let   $U$ be a sufficiently small open neighbourhood of $w^0$
in $\C^{n+m}$. The fibres of $\pi_w$ over $Q_0 \cap U$ have the dimension $n-1$ so the preimage $\pi_w^{-1}(Q_0 \cap U)$ contains an open piece (of dimension $2n-1$) of $M^\C$. Since $M^\C$ is irreducible,  we conclude by the uniqueness theorem that $M^\C \subset Z$ globally.
\end{proof}

The first step of proof consists of the following.

\begin{lemma}
 \label{boundary}
We have
\begin{itemize}
\item[(a)] $\Sigma$ is an open subset of $M^\C$.
\item[(b)] The boundary of $\Sigma$ is contained in a complex hypersurface in $M^\C$.
\end{itemize}
\end{lemma}

\begin{proof}  (a) The fact that  the set $\Sigma$ is open in $ M^\C_*$  follows immediately because the defining functions of a complex variety  $Q_{w}$ depend continuously on the parameter $w$. 

(b) We are going to describe the boundary of $\Sigma$. Let $(z^k,w^k)$ , $k=1,2,...,$ be a sequence of points from $\Sigma$ converging to some $(z^0,w^0) \in  M^\C_*$.  Every Segre variety $Q_{{w}}$  (for $w = w^0$ or $w = w^k$) is a complex analytic subset of dimension $n-1$ containing the origin. Assume that $(z^0,w^0)$ does not belong to $\Sigma$, that is, $(z^0,w^0)$ is a boundary point of $\Sigma$.  
The point  $(z^0,w^0)$ is a regular point for $M^\C$, and the point $z^0$ is a regular point of the Segre variety $Q_{{w^0}}$; 
we may assume that the same holds for every $(z^k,w^k)$.  
For $w = w^0$ or $w = w^k$ denote by $K(w)$ the unique irreducible component of $Q_{w}$ containing $z^0$ or $z^k$ respectively. It follows from the definition of $\Sigma$  that $K(w^0)$ contains the origin, while $K(w^k)$ does not contain the origin, $k=1,2,...$. 
The limit as $k\to\infty$ (with respect to  the Hausdorff distance) of the sequence $\{K(w^k)\}$  of complex analytic subsets is an $(n-1)$ dimensional complex analytic subset containing $K(w^0)$ as an irreducible component. Indeed, this is true in a neighbourhood of the point $z^0$ and then holds globally by the uniqueness theorem for irreducible complex analytic subsets.

We use the notation $z = (z',z_n,z'') = (z_1,...,z_{n-1},z_n,z_{n+1},...,z_{n+m})$. Performing a complex linear change of coordinates in $\C^{n+m}(z)$ if necessary,
we can assume that the intersection of $Q_{{w^0}}$ with 
the complex linear subspace $ \{ z: z' = 0' \} $ is a discrete set. Denote by $\D(z_{i_1},...,z_{i_l})$ the unit polydisc $\{ \vert z_{i_j} \vert < 1, j=1,...,l \}$ in the space 
 $\C(z_{i_1},...,z_{i_l})$.   Choose $\delta > 0$ small enough such that 
\begin{equation}
\label{e.fibre}
\{ z: (0',z''): z'' \in \delta \D(z_n,z'') \} \cap Q_{w^0} = \{ 0 \} .
\end{equation}
Using the notation $w = ('w,''w)$, where $'w = (w_{i_1},...,w_{i_n})$, choose a suitable complex affine subspace 
$ \{ ''w = {''w}^0 \}$ of $\C^{n+m}$ of dimension $n$  such that the canonical projection of $Q_0 \subset \C^{n+m}(w)$ on $\C^n('w)$ is proper. Recall that $\dim Q_0=n$. Shrinking $\delta$, one can assume that 
 $$
 Q_0 \cap \{ w: {'w} = {'w}^0, {''w} \in {''w}^0 + \delta\D({''w}) \} = \{ w^0 \}.
 $$
 Denote by $\pi$ the projection
 $$\pi: (z,w) \mapsto (z', {'w}).$$
 Then the intersection $\pi^{-1}(0',{'w}^0) \cap M^\C$ is discrete (we use here Lemma \ref{Q0}).
 Thus, there exist small enough neighbourhoods $U'$ of $0'$ in $\C^{n-1}(z')$,  $'V$ of 
 $'w^0$ in $\C^{n}('w)$, and $\delta > 0$ such that the restriction 
 $$\pi: M^\C \cap (U' \times \delta\D(z_n,z'') \times 'V  \times  (''w^0+\delta \D(''w)))  \longrightarrow U'  \times 'V$$
 is a proper map. This means that we have the following defining equations for $M^\C \cap (U' \times \delta\D(z_n,z'') \times 'V  \times  (''w^0+\delta \D(''w)))$:

\begin{eqnarray}\label{gamma1}
\left\{ (z,w) :  \Phi_{I}(z','{\overline w})(z_n,z'', {''w}):=  \sum_{ \vert J  \vert  \le d}  \phi_{IJ} (z',
'\overline {w})(z_n, z'', ({''w}-{''w}^0))^J  = 0,  \, \, \,  \vert I  \vert = d \right\},
\end{eqnarray}
where the coefficients $ \phi_{IJ}(z','{\overline w})$ are holomorphic in $(U' \times 'V)$.  The Segre varieties are obtained by fixing $'w$ in the above equations:
\begin{eqnarray}\label{Q1}
Q_{w}= \left\{ z : \Phi_{I}(z','{\overline w})(z_n,z'',{''w}):=  \sum_{ \vert J  \vert  \le d}  \phi_{IJ}(z','{\overline w}) (z_n,z'',
({''w} - {''w}^0))^J  = 0,  \, \, \,  \vert I  \vert = d \right\} .
\end{eqnarray}
Note that 
$ \phi_{I0} (0','{\overline w}) = 0$ for all $I$ and 
all $'w$ because every Segre variety contains the origin.

Denote by $\pi_j$ the projection
 $$\pi_j: (z,w) \mapsto (z',z_j,w), \,\,\, j=n,n+1,...,n+m .$$
 Then the restrictions 
 $$\pi_j: Q_{w} \cap (U' \times \delta\D(z_n,z'') )  \longrightarrow U' \times \delta \D(z_j)$$
are proper for every $w \in 'V \times (''w^0 + \delta \D(''w))$. The  image $\pi_j(Q_w)$ is a complex hypersurface in $U' \times \delta \D(z_j)$ with a proper projection on $U'$. Hence

 \begin{eqnarray}\label{Q2}
\pi_j(Q_{w})= \left\{ (z',z_j):  P_j(z', '\overline {w})(z_j):=z_j^{d_j} + a_{j d_j-1}(z', '{\overline w})z_j^{d_j-1}+ ...+ a_{j 0}(z', '{\overline w}) = 0 \right\} ,
\end{eqnarray}
 where the coefficients $a_{js}$ are holomorphic in $(z', '{\overline w}) \in U' \times V$.  Indeed, the equations (\ref{Q2}) are obtained from the equations (\ref{Q1}) by the standard elimination construction using the resultants of pseudopolynomials  $\Phi_{I}$ from (\ref{Q1}), see \cite{Ch}. This assures the holomorphic dependence of the coefficients with respect to the parameter $'{\overline w}$. Note that the first coefficient of each $P_j$  can be made equal to 1 since the projections are proper.
 
Recall that  $K(w^k)$ does not contain the origin in $\C^{n+m}$  for $k = 1,2,...$. Therefore, for each $k$ there exists some 
$j\in\{n, n+1, \dots, n+m\}$ such that the $z_j$ coordinate of some point of the fibre $\pi^{-1}(0') \cap K(w^k)$ does not vanish.
After passing to a subsequence and relabeling the coordinates one can  assume that  for every $k = 1,2,...,$ the $z_n$-coordinate of some point of the fibre $\pi^{-1}(0') \cap K(w^k)$ does not vanish.  The $z_n$-coordinate  of every point of the fibre $\pi^{-1}(0') \cap K(w^k)$ satisfies the equation  $(\ref{Q2})$ with  $j= n$, $z' = 0'$ and $w = w^k$. 
 Hence, for every $k$  this equation admits a nonzero solution.  Passing again to a subsequence one can also assume that there exists $s$ such that for every $k = 1,2,...,$  one has $a_{n s}(0','{\overline w}^k) \neq 0$. 
 
  Set $z' = 0$ in (\ref{Q2}).  Consider $(\zeta,'w) \in \C \times 'V$ satisfying the equation 
 
 \begin{eqnarray}
 \label{algebr}
\zeta^{d_n} + a_{n d_n-1}(0', '{\overline w})\zeta^{d_n-1}+ ...+ a_{n 0}(0', '{\overline w}) = 0  .
\end{eqnarray}
This equation defines a $d_n$-valued algebroid function $'w \mapsto \zeta('w)$. Given $'w \in 'V$ the algebroid function $\zeta$ 
associates to it the set $\zeta('w) = \{ \zeta_1('w),...,\zeta_s('w))\}$, $s = s('w) \le d_n$, of distinct roots of the equation (\ref{algebr}).  Let $j$ be the smallest index such that the coefficient $a_{nj}(0', '{\overline w})$ does not vanish identically. Dividing (\ref{algebr}) by $\zeta^j$ we obtain

 \begin{eqnarray}\label{Q3}
\zeta^{d_n-j} + a_{n d_n-1}(0', '{\overline w})\zeta^{d_n-j-1}+ ...+ a_{n j}(0', '{\overline w}) = 0 .  
\end{eqnarray} 
Every non-zero value of $ \zeta$ satisfies this equation.

For each $w = w^k$, $k=1,2,...$ or $w = w^0$ the fibre $ \{ (z',z_n): z' = 0 \}  \cap \pi_n(K(w))$ is a finite set $ \{ p^1(w),...,p^l(w) \}$, $l = l(k)$ in $\C^n(z',z_n)$. Each non-vanishing $n$-th coordinate $p^\mu_n(w^k)$, $k=1,2,...$ is a value of the algebroid function $ \zeta$ at $'w^k$. By our assumption  we have $p^\nu_n(w^k)  \neq 0$   for some $\nu$ and every $k=1,2,...$; one can assume that $\nu$ is the same for all $k$. These $p^\nu_n(w^k)$ satisfy  the equation (\ref{Q3}) with $'w = 'w^k$. On the other hand, it follows 
by~\eqref{e.fibre} that  the fibre $ \{ (z',z_n): z' = 0 \}  \cap \pi_n(K(w^0))$ is the singleton $\{ 0 \}$ in $\C^n(z',z_n)$; hence  $p^\mu_n(w^0) = 0 $ for all $\mu$.   The sequence $(p_n^\nu(w^k))$, $k=1,2,....$ tends to some $p_n^\nu(w^0)$ as $w^k  \longrightarrow w^0$.  Therefore, every such $p^\nu_n(w^0)$ also satisfies ( \ref{Q3}) with $'w = 'w^0$. But $p^\mu_n(w^0) = 0 $ for all $\mu$ and, in particular, $p_n^\nu(w^0) = 0$.  This means that $a_{nj}(0', '{\overline w}^0) = 0$.

Thus the boundary of $ \Sigma$ in $M^\C_*$ is contained in the complex analytic hypersurface 
$$A_1 = (z,w)  \in M^\C_*: a_{nj}(0', '{\overline w}) = 0  \} = 0.$$
The union $A_2 = A_1 \cup M^\C_{sing} \cup (M^\C \setminus M^\C_*)$ is a complex hypersurface in $M^\C$ with the following property: if $(z,w) \in M^\C$ is close enough to $(z^0,w^0)$ and belongs to the closure of $\Sigma$ but 
does not belong to $\Sigma$, then $(z,w) \in A_2$. \end{proof}

The second step of the proof is given by the following

 \begin{lemma}\label{l.compl}
 The complement of $\overline\Sigma$ has  a nonempty interior in $M^\C$.
  \end{lemma}
 \begin{proof} We begin with the choice of a suitable point $w^*$.  First fix any point 
 $(z^*, w^*) \in M^\C_*$.  We can assume that the rank of the projection $ \pi_w$ is maximal and is equal to $n$ in a neighbourhood $O$ of  $(z^*,w^*)$ in $M^\C_*$; denote by $S$ the image $\pi_w(O)$. As above, we use the notation $w = ('w,''w)$ and  suppose that the projection $ \sigma: S  \ni ('w,''w) \longrightarrow {'w}$ is one-to-one on  a neighbourhood $'W$ of $'w^*$ in $ \C^n$.
Let $l$ be the maximal number of components of $Q_{w}$ for $w$ with $'w  \in {'W}$, and let $w^*$ be such that $Q_{ {w^*}}$ has exactly $l$ geometrically distinct components.  One can assume that a neighbourhood ${'W}$ of $'w^*$ is chosen
such that $Q_w$ has exactly $l$ components  for all $w  \in  \sigma^{-1}({'W})  \subset S$. Let $K_1(w), \dots, K_l(w)$ be the irreducible components 
of $Q_{w}$. Note that the components $K_j(w)$ depend continuously on $w$.

Consider the sets $F_j = \{ 'w \in {'W} : 0 \in K_j( \sigma^{-1}('w)) \}$. Every set $F_j$ is closed in $'W$. 
Since $0\in Q_w$ for every $w \in S$ (by Lemma~\ref{Q0}), we have
$\cup_j F_j = {'W}$. Therefore, one of this sets, say, $F_1$, has a nonempty interior in $'W$. This means that there 
exists a small ball $B$ in $ \C^n('w)$ 
centred at some $'\tilde w$ such that $K_1(w)$ contains $0$ for all $'w \in B \cap {'W}$. 
Choose a regular point $\tilde z$ in $K_1(  \tilde w)$ where $ \tilde w =\sigma^{-1}('\tilde w))$. Then for every 
$(z,w)  \in M^\C$ near $(\tilde z, \tilde w)$, we have $z \in K_1(w)$, i.e., $(z,w)\notin \Sigma$. Hence, 
the complement of $\overline\Sigma$ has a nonempty interior. \end{proof}

Now  by Lemma~\ref{boundary}(b) and the Remmert-Stein theorem the closure 
$\overline\Sigma$ of $\Sigma$  coincides with an irreducible component of  $M^\C$. 
Since the complexification $M^\C$ is irreducible, the closure $\overline\Sigma$ of $\Sigma$ coincides  with the whole $M^\C$. This contradiction with Lemma~\ref{l.compl} concludes the proof of Proposition  \ref{empty} and proves Theorem  \ref{DicritTheo}.

\

\end{document}